\documentclass[12pt,leqno]{amsart}
\usepackage{euscript, amssymb, amsmath, amsthm}
\usepackage{epsfig}
\usepackage{graphicx}
\usepackage{caption}
\usepackage{longtable}
\usepackage{dcolumn}
\usepackage{setspace}
\usepackage[most]{tcolorbox}
\definecolor{webred}{rgb}{0.75,0,0}
\definecolor{webgreen}{rgb}{0,0.75,0}
\definecolor{refkey}{gray}{0.75}
\usepackage[pagebackref=true, colorlinks=true, citecolor=blue]{hyperref}

\numberwithin{equation}{section}

\newtheorem{theo}{Theorem}[section]
\newtheorem{lem}{Lemma}[section]

\newtheorem{Def}[theo]{Definition}
\theoremstyle{remark}
\newtheorem{rem}{Remark}[section]

\newcommand{\de}{\delta}

\newcommand{\ep}{\varepsilon}
\def\R{{\mathbb{R}}}

\def\d{\displaystyle}
\def\e{{\varepsilon}}
\def\p{\partial}
\def\G{G_1}
\def\GG{\tilde{G}_1}
\def\H{G_2}
\def\HH{\tilde{G}_2}

\date{}

\subjclass[2010]{35L71,  35B44}
\keywords{blow-up, critical curve, lifespan, nonlinear wave equations, semilinear weakly coupled system, scale-invariant damping, time-derivative nonlinearity.}

\tcbset{
    frame code={}
    center title,
    left=0pt,
    right=0pt,
    top=0pt,
    bottom=0pt,
    colback=gray!10,
    colframe=white,
    width=\dimexpr\textwidth\relax,
    enlarge left by=0mm,
    boxsep=5pt,
    arc=0pt,outer arc=0pt,
    }

\begin{document}

\title[Improvement on the blow-up for  a weakly coupled wave equations  - Invariant case]{Improvement on the blow-up for  a weakly coupled wave equations  with  scale-invariant damping and mass and time derivative nonlinearity}
\author[M. Hamouda  and M. A. Hamza]{Makram Hamouda$^{1}$ and Mohamed Ali Hamza$^{1}$}
\address{$^{1}$ Department of Basic Sciences, Deanship of Preparatory Year and Supporting Studies, Imam Abdulrahman Bin Faisal University, P. O. Box 1982, Dammam, Saudi Arabia.}

\medskip

\email{mmhamouda@iau.edu.sa (M. Hamouda)} 
\email{mahamza@iau.edu.sa (M.A. Hamza)}

\pagestyle{plain}


\maketitle
\begin{abstract}

An improvement of \cite{Palmieri} on the blow-up region and the lifespan estimate of a weakly coupled system of wave equations with damping and mass in the \textit{scale-invariant case} and with time-derivative nonlinearity is obtained in this article. Indeed, thanks to a better understanding of the dynamics of the solutions, we give here a better characterization of the blow-up region. Furthermore, the techniques used in this article may be extended to other systems and interestingly they simplify the proof of the blow-up result in \cite{Our3} which is concerned with the single wave equation in the same context as in the present work.

\end{abstract}


\section{Introduction}
\par\quad

The weakly coupled system of semilinear wave equations in the presence of  damping and mass terms in scale-invariant case with time derivative nonlinearity reads as follows:
\begin{align}\label{G-sys}
\begin{cases}\d u_{tt}-\Delta u +\frac{\mu_1}{1+t}u_t+\frac{\nu_1^2}{(1+t)^2}u=|\partial_t v|^p, &   
\quad (x,t) \in \R^N\times[0,\infty), \vspace{.1cm}\\
\d v_{tt}-\Delta v +\frac{\mu_2}{1+t}v_t+\frac{\nu_2^2}{(1+t)^2}v=|\partial_t u|^q, &  \quad (x,t) \in \R^N\times[0,\infty),\\
u(x,0)=\e f_1(x), \ \ v(x,0)=\e f_2(x), & \quad x\in \mathbb{R}^N, \\ u_t(x,0)=\e g_1(x), \ \  v_t(x,0)=\varepsilon g_2(x), & \quad x\in \mathbb{R}^N,
\end{cases}
\end{align} 
where $\mu_1,\mu_2, \nu^2_1$ and $\nu^2_2$ are  nonnegative constants.  The positive parameter $\e$ characterizes the size of the initial
data,  and $ f_1, f_2,g_1$ and $g_2$ are positive functions assumed to be compactly supported on  $B_{\R^N}(0,R), R>0$.\\
Naturally, along this article , we suppose that $p, q>1$.

First, we recall  the Glassey exponent $p_G$ which is given by
\begin{equation}\label{Glassey}
p_G=p_G(N):=1+\frac{2}{N-1}.
\end{equation}
It is well-known  that the aforementioned  critical value, $p_G$, characterizes the threshold between the global existence ($p>p_G$) and the nonexistence ($p \le p_G$) regions; see e.g. \cite{Hidano1,Hidano2,John1,Sideris,Tzvetkov,Zhou1}.\\

In this paragraph, we  recall some results related the massless case for a single equation. We start by mentioning  the blow-up result for the solution of a single equation inherited from \eqref{G-sys} without mass term. Indeed,   Lai and Takamura  showed in \cite{LT2} an upper bound estimate of the lifespan. Later, in \cite{Palmieri}, Palmieri and Tu  enhanced this result  by extending the blow-up region for $p$ in the case of a single equation (with mass term). More precisely, they obtain a blow-up result for $p \in (1, p_G(N+\sigma(\mu,0))]$ where 
\begin{equation}\label{sigma}
\sigma=\sigma(\mu,\nu):=\left\{
\begin{array}{lll}
\mu+1 -\sqrt{\de}& \textnormal{if} &\de \in [0,1),\\
\mu & \text{if} &\de \ge 1,
\end{array}
\right.
\end{equation}
and $\de$ is given by \eqref{delta} below.  Note that the result in \cite{Palmieri} was recently improved in \cite{Our2} by extending the upper bound for $p$ from $p_G(N+\sigma(\mu,0))$ to $p_G(N+\mu)$.  \\

Recently, in  \cite{Our3}, the  case of the scale-invariant  damped equation with mass and  time-derivative nonlinearity was studied; this reads  
\begin{equation}
\label{T-sys-bis-old}
\left\{
\begin{array}{l}
\d u_{tt}-\Delta u+\frac{\mu}{1+t}u_t+\frac{\nu^2}{(1+t)^2}u=|u_t|^p, 
\quad \mbox{in}\ \R^N\times[0,\infty),\\
u(x,0)=\e f(x),\ u_t(x,0)=\e g(x), \quad  x\in\R^N.
\end{array}
\right.
\end{equation}
Let us introduce the following  quantity, that we assume to be  positive,
\begin{equation}\label{delta}
\de=\de(\mu,\nu):=(\mu-1)^2-4\nu^2.
\end{equation}
Furthermore, we define
\begin{equation}\label{sigma}
\sigma=\sigma(\mu,\nu):=\left\{
\begin{array}{lll}
\mu +1-\sqrt{\delta}& \textnormal{if} &\delta \in [0,1),\\
\mu & \text{if} &\delta \ge 1.
\end{array}
\right.
\end{equation}
Using a functional approach, a blow-up result is proven in \cite{Our3} for  \eqref{T-sys-bis-old}, improving thus the one obtained in \cite{Palmieri}. In fact,  the blow-up interval, $p \in (1, p_G(N+\sigma)]$ ($\sigma$ is given by \eqref{sigma}),  is ameliorated in comparison with \cite{Palmieri},  to  show the blow-up inside the region  $p \in (1, p_G(N+\mu)]$, with $\de \in (0,1)$. However, for $\de \ge 1$, the two results, in  \cite{Our3} and \cite{Palmieri}, are the same.  In relationship with these works, we also mention the articles \cite{Our2,LT2} where the massless case is investigated.\\

Now, letting $\mu_1 = \mu_2 = \nu_1=\nu_2= 0$ in \eqref{G-sys} we find the following coupled system:
\begin{align}\label{G-sys-0}
\begin{cases} u_{tt}-\Delta u =|\partial_t v|^p, &   
 (x,t) \in \R^N\times[0,\infty), \vspace{.1cm}\\
 v_{tt}-\Delta v =|\partial_t u|^q, &   
 (x,t) \in \R^N\times[0,\infty), \vspace{.1cm}\\
u(x,0)=\e f_1(x), \ \ v(x,0)=\e  f_2(x), & x\in \mathbb{R}^N, \\ u_t(x,0)=\e g_1(x), \ \  v_t(x,0)=\varepsilon g_2(x), & x\in \mathbb{R}^N.
\end{cases}
\end{align}
For the global existence of solutions to  \eqref{G-sys-0}, we refer the reader to \cite{Kubo}. However, the blow-up of \eqref{G-sys-0} has been the subject of several works; see e.g. \cite{Dao-Reissig,Ikeda-sys,Palmieri1,Palmieri-Takamura-arx}.  More precisely, the critical (in the sense of interface between blow-up and global existence) curve for $p,q$ is given by
\begin{equation}\label{crit-curve-0}
\Upsilon(N,p,q):=\max (\Lambda(N,p,q), \Lambda(N,q,p))=0,
\end{equation}
where 
\begin{equation}\label{Lambda}
\d \Lambda(N,p,q):= \frac{p+1}{pq-1}-\frac{N-1}{2}.
\end{equation}
Under some assumptions, the solution $(u,v)$ of \eqref{G-sys-0} blows up in finite time $T(\e)$ for small initial data (of size $\e$),  namely 
\begin{equation}\label{Teps}
T(\e) \le \left\{
\begin{array}{lll}
C \e^{-\Upsilon(N,p,q)}& \text{if}&\Upsilon(N,p,q)>0,\\
\exp(C \e^{-(pq-1)})& \text{if}&\Upsilon(N,p,q)=0, \ p \neq q,\\
\exp(C \e^{-(p-1)})& \text{if}&\Upsilon(N,p,q)=0, \ p = q.
\end{array}
\right.
\end{equation}

In the context of the present work,  Palmieri and Tu  \cite{Palmieri} proved   a blow-up result for the system  \eqref{G-sys}. More precisely, the authors in \cite{Palmieri} proved that there is blow-up for the system  \eqref{G-sys} for $p,q$  satisfying
\begin{equation}\label{blow-up-reg}
\Omega(N,\sigma_1,\sigma_2,p,q):=\max (\Lambda(N+\sigma_1,p,q), \Lambda(N+\sigma_2,q,p)) \ge 0,
\end{equation}
where $\Lambda$ is given by \eqref{Lambda} and $\sigma_i=\sigma(\mu_i,\nu_i), i=1,2$ is given by  \eqref{sigma}.\\
Indeed, for small initial data (of size $\e$), the solution $(u,v)$ of \eqref{G-sys} blows up in finite time $T(\e)$ that is bounded as 
\begin{equation}\label{Teps}
T(\e) \le \left\{
\begin{array}{lll}
C \e^{-\Omega(N,\sigma_1,\sigma_2,p,q)}& \text{if}&\Omega(N,\sigma_1,\sigma_2,p,q)>0,\\
\exp(C \e^{-(pq-1)})& \text{if}&\Omega(N,\sigma_1,\sigma_2,p,q)=0, \\
\exp(C \e^{-\min \left(\frac{pq-1}{p+1},\frac{pq-1}{q+1}\right)})& \text{if}&\Lambda(N+\sigma_1,p,q)=\Lambda(N+\sigma_2,p,q)=0.
\end{array}
\right.
\end{equation}

Concerning the case of coupled equations, namely \eqref{G-sys}  without mass terms ($ \nu_1=\nu_2= 0$), an improvement was obtained in \cite{Our4} for the blow-up results and the lifespan. The results in \cite{Our4} are ameliorating the ones in \cite{Palmieri}.\\

In this article, we refine the results in \cite{Palmieri} when at least one of the coefficients $\delta_i, i=1,2$, is in $[0,1)$. In fact, for positive values of $\delta_i$  ($i=1,2$), we extend the  results obtained in  \cite{Palmieri} to show that the new blow-up region  does not depend on the mass parameters $\nu_i, i=1,2$ (and hence nor on $\sigma_i, i=1,2$). A similar observation was concluded for the case of a one equation that we studied in \cite{Our3}, and the aim here is to extend this examination to the case of the coupled system  \eqref{G-sys}. But, although the obtaining of similar results for damped coupled systems with mass terms is predictable, the situation is somehow more delicate. Indeed, the techniques used in \cite{Our3} are now longer effective for the system \eqref{G-sys}. To overcome this difficulty, we first write the linear problem associated with \eqref{G-sys} which reduces in this case to the following single equation: 
\begin{align}\label{G-sys-lin}
\d w^L_{tt}-\Delta w^L +\frac{\mu_i}{1+t}w^L_t+\frac{\nu_i^2}{(1+t)^2}w^L=0, 
\end{align} 
where $w$ stands for $u^L$ or $v^L$, the solutions of the linear problem associated with \eqref{G-sys}.\\
Clearly, the equation  \eqref{G-sys-lin} is invariant under the following transformation:
\begin{align}\label{invar}
\tilde{w}^L(x,t)=w^L(\alpha x, \alpha(1+t)-1), \ \alpha>0.
\end{align} 
Taking advantage of the aforementioned invariance properties, we refine our choice for a functional family that is indexed by a positive parameter $\eta$.  This judicious choice implies a better description of  the dynamics of the solution of  \eqref{G-sys}. More precisely, we obtain, for $\eta$ large enough, the coercivity of the functional that will be introduced later on to show the blow-up results.

Thanks to the above observations, we finally enhance the result on the blow-up region, defined by \eqref{blow-up-reg} and obtained in \cite{Palmieri},   by showing that the critical curve is characterized by
\begin{equation}\label{blow-up-reg-imp}
\Omega(N,\mu_1,\mu_2,p,q)=\max (\Lambda(N+\mu_1,p,q), \Lambda(N+\mu_2,q,p)) \ge 0,
\end{equation}
where $\Lambda$ is given by \eqref{Lambda}.\\


The  outline of this article is presented as follows. First, we introduce in Section \ref{sec-main}   the weak formulation of (\ref{G-sys}) in the energy space. Then, in the same section, we state our main result.  Some technical lemmas are proven in Section \ref{aux}. Finally, we show  the proof of the main result in Section \ref{sec-ut}.

\section{Main Result}\label{sec-main}
\par

The aim of this  section is  to state our main result for which  we will write  the equivalent of the system (\ref{G-sys})   in the corresponding energy space. More precisely, the weak formulation associated with  (\ref{G-sys}) reads  as follows:
\begin{Def}\label{def1}
 We say that $(u,v)$ is an energy  solution of
 (\ref{G-sys}) on $[0,T)$
if
\begin{displaymath}
\left\{
\begin{array}{l}
u,v\in \mathcal{C}([0,T),H^1(\R^N))\cap \mathcal{C}^1([0,T),L^2(\R^N)), \vspace{.1cm}\\
  \ u_t \in L^q_{loc}((0,T)\times \R^N), \ v_t \in L^p_{loc}((0,T)\times \R^N)
 \end{array}
  \right.
\end{displaymath}
satisfies, for all $\Phi, \tilde{\Phi} \in \mathcal{C}_0^{\infty}(\R^N\times[0,T))$ and all $t\in[0,T)$, the following equations:
\begin{equation}
\label{energysol2}
\begin{array}{l}
\d\int_{\R^N}u_t(x,t)\Phi(x,t)dx-\int_{\R^N}u_t(x,0)\Phi(x,0)dx -\int_0^t  \int_{\R^N}u_t(x,s)\Phi_t(x,s)dx \,ds \vspace{.2cm}\\
\d+\int_0^t  \int_{\R^N}\nabla u(x,s)\cdot\nabla\Phi(x,s) dx \,ds +\int_0^t  \int_{\R^N}\frac{\mu_1}{1+s}u_t(x,s) \Phi(x,s)dx \,ds\vspace{.2cm}\\
\d +\int_0^t\int_{\R^N}\frac{\nu_1^2}{(1+s)^2}u(x,s)\Phi(x,s)dx\,ds=\int_0^t \int_{\R^N}|v_t(x,s)|^p\Phi(x,s)dx \,ds,
\end{array}
\end{equation}
and
\begin{equation}
\label{energysol3}
\begin{array}{l}
\d\int_{\R^N}v_t(x,t)\tilde{\Phi}(x,t)dx-\int_{\R^N}v_t(x,0)\tilde{\Phi}(x,0)dx -\int_0^t  \int_{\R^N}v_t(x,s)\tilde{\Phi}_t(x,s)dx \,ds\vspace{.2cm}\\
\d +\int_0^t  \int_{\R^N}\nabla v(x,s)\cdot\nabla\tilde{\Phi}(x,s) dx \,ds +\int_0^t  \int_{\R^N}\frac{\mu_2}{1+s}v_t(x,s) \tilde{\Phi}(x,s)dx \,ds\vspace{.2cm}\\
\d +\int_0^t\int_{\R^N}\frac{\nu_2^2}{(1+s)^2}v(x,s)\tilde{\Phi}(x,s)dx\,ds=\int_0^t \int_{\R^N}|u_t(x,s)|^q\tilde{\Phi}(x,s)dx \,ds,
\end{array}
\end{equation}
together with the conditions $u(x,0)=\varepsilon  f_1(x)$ and $v(x,0)=\varepsilon  f_2(x)$ being satisfied in $H^1(\mathbb{R}^N)$.\\
After some elementary computations, \eqref{energysol2} and \eqref{energysol3} can be written, respectively, in the following way
\begin{equation}
\begin{array}{l}\label{energysol2-bis}
\d \int_{\R^N}\big[u_t(x,t)\Phi(x,t)- u(x,t)\Phi_t(x,t)+\frac{\mu_1}{1+t}u(x,t) \Phi(x,t)\big] dx \vspace{.2cm}\\
\d \int_0^t  \int_{\R^N}u(x,s)\left[\Phi_{tt}(x,s)-\Delta \Phi(x,s) -\frac{\partial}{\partial s}\left(\frac{\mu}{1+s}\Phi(x,s)\right)+\frac{\nu_1^2}{(1+s)^2}\Phi(x,s)\right]dx \,ds\vspace{.2cm}\\
\d =\int_{0}^{t}\int_{\R^N}|u_t(x,s)|^p\Phi(x,s)dx \, ds + \e \int_{\R^N}\big[-f_1(x)\Phi_t(x,0)+\left(\mu_1 f_1(x)+g_1(x)\right)\Phi(x,0)\big]dx,
\end{array}
\end{equation}
and
\begin{equation}
\begin{array}{l}\label{energysol3-bis}
\d \int_{\R^N}\big[v_t(x,t)\tilde{\Phi}(x,t)- v(x,t)\tilde{\Phi}_t(x,t)+\frac{\mu_2}{1+t}u(x,t) \tilde{\Phi}(x,t)\big] dx \vspace{.2cm}\\
\d \int_0^t  \int_{\R^N}v(x,s)\left[\tilde{\Phi}_{tt}(x,s)-\Delta \tilde{\Phi}(x,s) -\frac{\partial}{\partial s}\left(\frac{\mu_2}{1+s}\tilde{\Phi}(x,s)\right)+\frac{\nu_2^2}{(1+s)^2}\tilde{\Phi}(x,s)\right]dx \,ds\vspace{.2cm}\\
\d =\int_{0}^{t}\int_{\R^N}|v_t(x,s)|^p\tilde{\Phi}(x,s)dx \, ds + \e \int_{\R^N}\big[-f_2(x)\Phi_t(x,0)+\left(\mu_2 f_2(x)+g_2(x)\right)\tilde{\Phi}(x,0)\big]dx. 
\end{array}
\end{equation}
\end{Def}

\begin{rem}\label{rem-supp}
Since $f$ and $g$ are supported on $B_{\R^N}(0,R)$, one can see that $\mbox{\rm supp}(u), \mbox{\rm supp}(v) \ \subset\{(x,t)\in\R^N\times[0,\infty): |x|\le t+R\}$. Consequently, one  can choose any test function $\Phi$  which is not necessarily compactly supported.
\end{rem}

In the following, we state the  main result of this article.
\begin{theo}
\label{blowup}
Let $p, q >1$. For  $i=1,2$, let $\mu_i, \nu_i^2 >0$,    with $\delta_i:=\de(\mu_i,\nu_i) >0$ (see \eqref{delta}), 
such that 
\begin{equation}\label{assump}
\Omega(N,\mu_1,\mu_2,p,q) \ge 0,
\end{equation}
where  $\Omega$ is defined by \eqref{blow-up-reg}.\\
Assume that  $f_1, f_2\in H^1(\R^N)$ and $g_1, g_2 \in L^2(\R^N)$ are non-negative functions which are compactly supported on  $B_{\R^N}(0,R)$,
and  do not vanish everywhere. Furthermore, we suppose that
\begin{equation}\label{hypfg}
  \frac{\mu_i-1-\sqrt{\de_i}}{2}f_i(x)+g_i(x) \ge 0, \quad i=1,2.
  \end{equation}
Let $(u,v)$ be an energy solution of \eqref{energysol2}-\eqref{energysol3} on $[0,T_\e)$ such that $\mbox{\rm supp}(u), {\rm supp}(v)\ \subset\{(x,t)\in\R^N\times[1,\infty): |x|\le t+R\}$. 
Then, there exists a constant $\e_0=\e_0( f_1,  f_2,g_1, g_2,N,R,p,q,\mu_1,\mu_2)>0$
such that $T_\e$ verifies
\begin{equation}\label{Teps1}
T(\e) \le \left\{
\begin{array}{lll}
C \e^{-\Omega(N,\mu_1,\mu_2,p,q)}& \text{if}&\Omega(N,\mu_1,\mu_2,p,q)>0,\\
\exp(C \e^{-(pq-1)})& \text{if}&\Omega(N,\mu_1,\mu_2,p,q)=0, \\
\exp(C \e^{-\min \left(\frac{pq-1}{p+1},\frac{pq-1}{q+1}\right)})& \text{if}&\Lambda(N+\mu_1,p,q)=\Lambda(N+\mu_2,q,p)=0.
\end{array}
\right.
\end{equation}
 where $C$ is a positive constant independent of $\e$ and $0<\e\le\e_0$.
\end{theo}

\begin{rem}
The blow-up result in Theorem \ref{blowup} exhibits the new region obtained for the critical curve for $p,q$ as a shift of the dimension by $\mu_1, \mu_2$. We believe that this new blow-up region delimitation coincides with the critical one. Of course, a rigorous confirmation should be proved by a global existence result.
\end{rem}

\begin{rem}
The result in Theorem \ref{blowup}  holds true for the following generalized system:
\begin{align}\label{G-sys-gen}
\begin{cases}\d u_{tt}-\Delta u +\left(b_1(t)+\frac{\mu_1}{1+t}\right)u_t+\left(c_1(t)+\frac{\nu_1^2}{(1+t)^2}\right)u=|\partial_t v|^p, &   
\ (x,t) \in \R^N\times[0,\infty), \vspace{.1cm}\\
\d v_{tt}-\Delta v +\left(b_2(t)+\frac{\mu_2}{1+t}\right)v_t+\left(c_2(t)+\frac{\nu_2^2}{(1+t)^2}\right)v=|\partial_t u|^q, &  \ (x,t) \in \R^N\times[0,\infty),
\end{cases}
\end{align} 
where $b_i(t), (1+t)c_i(t)$ belong to $L^1(0,\infty)$ and $\mu_i, \nu_i$ are such that $\delta_i >0$; $i=1,2$. \\
The proof of the generalized damping case \eqref{G-sys-gen} can be performed by mimicking  the one of Theorem \ref{blowup}   with the necessary modifications.
\end{rem}

\begin{rem}
The techniques used in this article can be of course adapted in other contexts. More precisely, the case of a single equation corresponding to \eqref{G-sys} can be simplified by taking into account the invariance property \eqref{invar} which is related to the introduction of the parameter $\eta$. Furthermore, one can use the aforementioned techniques to study systems like \eqref{G-sys} with mixed nonlinearities, this will be the subject of a forthcoming work which will somehow constitute  an extension of our previous works \cite{Our, Our4, Our2}, see also \cite{LT2,Palmieri-Takamura}.
\end{rem}

\section{Some auxiliary results}\label{aux}
\par

We first introduce a positive test function  which is defined as
\begin{equation}
\label{test11}
\psi_i^{\eta}(x,t):=\rho_i^{\eta}(t)\phi^{\eta}(x), \ i=1,2, \ \forall \ \eta >0,
\end{equation}
where
\begin{equation}
\label{test12}
\phi^{\eta}(x):=
\left\{
\begin{array}{ll}
\d\int_{S^{N-1}}e^{\eta x\cdot\omega}d\omega & \mbox{for}\ N\ge2,\vspace{.2cm}\\
e^{\eta x}+e^{-\eta x} & \mbox{for}\  N=1.
\end{array}
\right.
\end{equation}
Note that the function $\phi^{\eta}(x)$ is introduced in \cite{YZ06}   and $\rho_i^{\eta}(t)$, \cite{Palmieri1,Palmieri-Tu,Tu-Lin1,Tu-Lin},   is solution of 
\begin{equation}\label{lambda}
\frac{d^2 \rho_i^{\eta}(t)}{dt^2}-\frac{d}{dt}\left(\frac{\mu_i}{1+t}\rho_i^{\eta}(t)\right) + \left(\frac{\nu_i^2}{(1+t)^2}-\eta^2\right)\rho_i^{\eta}(t)=0, \ i=1,2.
\end{equation}
From the literature, it is well-known that the expression of  $\rho_i^{\eta}(t)$ is given by (see the Appendix for more details),
\begin{equation}\label{lmabdaK}
\rho_i^{\eta}(t)=(\eta (t+1))^{\frac{\mu_i+1}{2}}K_{\frac{\sqrt{\de_i}}2}(\eta(t+1)), \ i=1,2,
\end{equation}
where 
$$K_{\nu}(t)=\int_0^\infty\exp(-t\cosh \zeta)\cosh(\nu \zeta)d\zeta,\ \nu\in \mathbb{R}.$$
Furthermore, we have that the function $\phi^{\eta}(x)$ satisfies
\begin{equation*}\label{phi-pp}
\Delta\phi^{\eta}=\eta^2\phi^{\eta}.
\end{equation*}
One can easily see that the function $\psi_i^{\eta}(x, t)$ fulfills the following conjugate equation:
\begin{equation}\label{lambda-eq}
\partial^2_t \psi_i^{\eta}(x, t)-\Delta \psi_i^{\eta}(x, t) -\frac{\partial}{\partial t}\left(\frac{\mu_i}{1+t}\psi_i^{\eta}(x, t)\right) + \frac{\nu_i^2}{(1+t)^2}\psi_i^{\eta}(x, t)=0.
\end{equation}

In what follows and through this article, the constant  $C$  stands for any  generic positive number which may depend on the data ($p,q,\mu_i,N,R, f_i,g_i$)$_{i=1,2}$ but not on $\ep$ and whose  value may change from line to line. However, in some occurrences and when it is necessary, we will precise the dependence of the constant $C$ on the parameters involved in this work.\\

Now, we state without proof the following lemma which gives a useful estimate for the function $\psi_i^{\eta}(x, t)$.
\begin{lem}[\cite{YZ06}]
\label{lem1} Let  $r>1$.
Then, there exists a constant $C=C(\eta,N,R,r)>0$ such that
\begin{equation}
\label{psi}
\int_{|x|\leq t+R}\Big(\phi^{\eta}(x)\Big)^{r}dx
\leq Ce^{rt}(1+t)^{\frac{(2-r)(N-1)}{2}},
\quad\forall \ t\ge 0.
\end{equation}
\end{lem}

In order to show the blow-up result later on,  the following functionals are introduced here.
\begin{equation}
\label{F1def-old}
F_1^{\eta}(t):=e^{-\eta t}\int_{\R^N}u(x, t)\phi^{\eta}(x)dx, \quad F_2^{\eta}(t):=e^{-\eta t}\int_{\R^N}v(x, t)\phi^{\eta}(x)dx,
\end{equation}
and
\begin{equation}
\label{F2def-old}
\tilde{F}_1^{\eta}(t):=e^{-\eta t}\int_{\R^N}\p_tu(x, t)\phi^{\eta}(x)dx, \quad \tilde{F}_2^{\eta}(t):=e^{-\eta t}\int_{\R^N}\p_tv(x, t)\phi^{\eta}(x)dx,
\end{equation}
where $\eta$ is a positive constant that will be determined later on. We also define the multiplier $m_i(t)$ for $i=1,2$ as follows:
\begin{equation}\label{mu}
m_i(t)=(1+t)^{\mu_i}, \quad i=1,2.
\end{equation}
Hence, the next two lemmas give the first  lower bounds for $F_i^{\eta}(t)$ and $\tilde{F}_i^{\eta}(t)$, \ i=1,2, respectively.

\begin{lem}
\label{F1}
Assume that the assumption in Theorem \ref{blowup} holds. Then, we have
\begin{equation}
\label{F1postive}
F_i^{\eta}(t)\ge 0, 
\quad\text{for all}\ t \in [0,T),  \ i=1,2,
\end{equation}
for all $\eta \ge \eta_0$ where $\eta_0$ is given by
\begin{equation}\label{eta0}
\eta_0:=1+\max(|\nu_1|,|\nu_2|).
\end{equation}
\end{lem}
\begin{proof} 
Let $t \in [0,T)$. We first employ  Definition \ref{def1},   perform an integration by parts in space in the fourth term in the left-hand side of \eqref{energysol2} and then choose $\psi^{\eta}_1(x, t)$ as a test function\footnote{Note that  it is possible to consider here not compactly supported test
functions thanks to the support property of $u$. Indeed, it is sufficient to replace $\psi^{\eta}_1(x, t)$ by $\psi^{\eta}_1(x, t) \chi(x, t)$ where $\chi$ is compactly supported such that $\chi(x, t)\equiv 1$ on $\mbox{\rm supp}(u)$.},  we obtain that
\begin{equation}
\begin{array}{l}\label{eq4-bis}
\d m_1(t)\int_{\R^N}u_t(x,t)\psi^{\eta}_1(x,t)dx
-\e\int_{\R^N}g_1(x)\psi^{\eta}_1(x,0)dx \vspace{.2cm}\\
\d+\int_0^tm_1(s)\int_{\R^N}\left\{
u_t(x,s)\psi^{\eta}_1(x,s)-\eta^2u(x,s)\psi^{\eta}_1(x,s)\right\}dx \, ds \vspace{.2cm}\\
\d+\int_0^t  \int_{\R^N}\frac{\nu_1^2m_1(s)}{(1+s)^2}u(x,s) \psi^{\eta}_1(x,s)dx \,ds\vspace{.2cm}\\
\d=\int_0^tm_1(s)\int_{\R^N}|v_t(x,s)|^p\psi^{\eta}_1(x,s)dx \, ds,
\end{array}
\end{equation}
where $m_1(t)$ is defined by \eqref{mu}.

Using  the fact that $$\d \int_0^tm_1(s)\frac{d F_1^{\eta}}{ds}(s) ds=- \int_0^tm_1'(s)F_1^{\eta}(s) ds+m_1(t)F_1^{\eta}(t)-F_1^{\eta}(0),$$
and the definition of $F_1^{\eta}$,  the equation  \eqref{eq4-bis} gives
\begin{equation}
\begin{array}{ll}\label{eq5-1}
\d m_1(t)\left(\frac{d F_1^{\eta}}{dt}(t)+(1+\eta)F_1^{\eta}(t)\right)
-{\e}C^{\eta}_0(f_1,g_1)  & \d =\int_0^t\left\{\frac{\mu_1}{1+s}- \frac{\nu_1^2}{(1+s)^2}-\eta+\eta^2\right\}m_1(s)F_1^{\eta}(s) ds\vspace{.2cm}\\
&\d+\int_0^tm_1(s)\int_{\R^N}|v_t(x,s)|^p\psi^{\eta}_1(x,s)dx \, ds,
\end{array}
\end{equation}
where 
$$C^{\eta}_0(f_1,g_1):=\int_{\R^N}\left\{f_1(x)+g_1(x)\right\}\phi^{\eta}(x)dx.$$

Multiplying \eqref{eq5-1}   by $e^{(1+\eta)t}/m_1(t)$, we deduce after integrating  over $[0,t]$ that
\begin{equation}
\begin{array}{ll}\label{F1+}
\d F_1^{\eta}(t)
\ge F_1^{\eta}(0)e^{-(1+\eta)t}+{\e}C^{\eta}_0(f_1,g_1)\int_0^t \frac{ e^{(1+\eta)(s-t)}}{m_1(s)}ds \\ 
\d+  \int_0^t\frac{\mu_1 e^{(1+\eta)(s-t)}}{m_1(s)}\int_{0}^{s} \left\{\frac{\mu_1}{1+\eta}- \frac{\nu_1^2}{(1+\tau)^2}-\eta+\eta^2\right\}m_1(\tau)F_1^{\eta}(\tau) d\tau ds.
\end{array}
\end{equation}
Since $\eta \ge \eta_0$, then we have $\d - \nu_1^2-\eta+\eta^2>0$.\\
Thanks to \eqref{F1+} and the information that $F_1^{\eta}(0)>0$, we deduce that $F_1^{\eta}(0)>0, \forall \, t \ge 0$; see \cite[Sec. 3]{LT}.

Similarly, one can prove that $F_2^{\eta}(t)$ is bounded by zero from below  thanks to the fact that $\eta \ge \eta_0$.

This ends the proof of Lemma \ref{F1}.
\end{proof}

The next step consists in proving  the positivity of the functional  $\tilde{F}_i^{\eta}(t)$ which is subject of the following lemma.
\begin{lem}
\label{lemF2}
Under the  assumption as in Theorem \ref{blowup}, it holds that
\begin{equation}
\label{F2postive}
\tilde{F}_i^{\eta}(t)\ge 0,
\quad\text{for all}\ t \in [0,T), \ \eta \ge \eta_0, \ i=1,2,
\end{equation}
where $\eta_0$ is given by \eqref{eta0}.
\end{lem}

\begin{proof} 
Let $t \in [0,T)$. 
Using the definition of $F_1^{\eta}$ and  $\tilde{F}_1^{\eta}$, given respectively by \eqref{F1def} and  \eqref{F2def}, and the fact that
 \begin{equation}\label{def231}\d \frac{ d F_1^{\eta}}{dt}(t) +\eta F_1^{\eta}(t)= \tilde{F}_1^{\eta}(t),\end{equation}
 the equation  \eqref{eq5-1} yields
\begin{equation}
\begin{array}{l}\label{eq5bis11}
\d m_1(t)(\tilde{F}_1^{\eta}(t)+F_1^{\eta}(t))
-{\e}C^{\eta}_0(f_1,g_1) =\int_0^t\left\{\frac{\mu_1}{1+s}- \frac{\nu_1^2}{(1+s)^2}-\eta+\eta^2\right\}m_1(s)F_1^{\eta}(s) ds\vspace{.2cm}\\
\d+\int_0^tm_1(s)\int_{\R^N}|v_t(x,s)|^p\psi_1^{\eta}(x,s)dx \, ds.
\end{array}
\end{equation}
Differentiating the  equation \eqref{eq5bis11} in time and using \eqref{def231}, we obtain
\begin{equation}
\begin{array}{ll}\label{eq5bis1}
\d \frac{d}{dt} \left\{\tilde{F}_1^{\eta}(t)m_1(t)\right\} +m_1(t)\tilde{F}_1^{\eta}(t)
&\d  =\left\{- \frac{\nu_1^2}{(1+t)^2}+\eta^2\right\}m_1(t)F_1^{\eta}(t) \vspace{.2cm}\\
& \d +m_1(t)\int_{\R^N}|v_t(x,t)|^p\psi_1^{\eta}(x,t)dx.
\end{array}
\end{equation}
Thanks to the fact that $\eta \ge \eta_0$, we can ignore the right-hand side in   \eqref{eq5bis1} which is now positive.  Then, the identity \eqref{eq5bis1} yields
\begin{align}\label{F1+bis3}
\d \frac{d}{dt} \left\{\tilde{F}_1^{\eta}(t)m_1(t) e^t\right\} \ge 0.
\end{align}
An analogous estimate to \eqref{F1+bis3} can be derived for $F_2^{\eta}(t)$ as well.
Finally, since $\tilde{F}_i^{\eta}(0)>0,  i=1,2$, we conclude  the proof of Lemma \ref{lemF2}.
\end{proof}

\begin{rem}
One can note that the lower bounds  in Lemmas \ref{F1} and \ref{lemF2} are not optimal. However, the aforementioned results are sufficient  to prove our main result since we only need the positivity of $F_i^{\eta}(t)$ and $\tilde{F}_i^{\eta}(t), \ i=1,2$, for all $t>1$. To enhance these lower bounds, we will instead introduce new functionals as we will see in the  next lemmas.
\end{rem}

Although the computations in the rest of this article can be carried out for all $\eta \ge \eta_0$, we choose from now to set $\eta = \eta_0$ and ignore the dependence on $\eta$ for all the functions (already introduced or which will be later on) that will be subsequently used. For example, the function $\psi_i^{\eta}$ will be simply denoted by $\psi_i$, and the same for all the other functions (including the constants) unless otherwise specified.

\par
The following functionals can be now introduced and will be used subsequently in the proof of the blow-up criteria later on,
\begin{equation}
\label{F1def}
\G (t):=\int_{\R^N}u(x, t)\psi_1(x, t)dx, \quad \H (t):=\int_{\R^N}v(x, t)\psi_2(x, t)dx,
\end{equation}
and
\begin{equation}
\label{F2def}
\GG (t):=\int_{\R^N}\p_tu(x, t)\psi_1(x, t)dx, \quad \HH (t):=\int_{\R^N}\p_t v(x, t)\psi_2(x, t)dx.
\end{equation}

The aim of the next two lemmas is to prove that the functions $\e^{-1}G_i(t)$ and $\e^{-1}\tilde{G}_i(t)$ are  {\it coercive}. Indeed, as we will see later on in \eqref{F1postive} and \eqref{F2postive} below, we will improve the  lower bounds already obtained for the functionals $F_i(t)$ and $\tilde{F}_i(t)$. This improvement will be useful in the proof of the main result of this article.

 Although the techniques used here are somehow close  to the ones in our previous work \cite{Our2} (which studies the one single equation corresponding to \eqref{G-sys}), but, the situation is slightly different for the system \eqref{G-sys}. So, we will include  all the details about the proofs of the next two lemmas. However, we will only show the proofs for the solution $u$, and for $v$ the computations follow similarly.  
\begin{lem}
\label{F1-n}
Assume that the assumptions in Theorem \ref{blowup} hold. Let $(u,v)$ be an energy solution of \eqref{energysol2}-\eqref{energysol3}. Then, for $i=1,2$,  there exists $T_0=T_0(\mu_i,\nu_i, \eta_0)>1$ such that  we have
\begin{equation}
\label{F1postive}
G_i(t)\ge C_{G_i}\, \e, 
\quad\text{for all}\ t \in [T_0,T),
\end{equation}
where $C_{G_i}$ is a positive constant which depends on $f_i, g_i$, $N,R,\eta_0$ and $\mu_i, \nu_i$.
\end{lem}

\begin{proof} 
Let $t \in [0,T)$.  Replacing   $\Phi$ by $\psi_1$  in \eqref{energysol2-bis} and employing \eqref{lambda-eq} yield
\begin{equation}
\begin{array}{l}\label{energysol2-bis2}
\d \int_{\R^N}\bigg[u_t(x,t)\psi_1(x,t)- u(x,t)\frac{\partial \psi_1}{\partial t}(x,t)+\frac{\mu_1}{1+t}u(x,t) \psi_1(x,t)\bigg] dx \vspace{.2cm}\\
\d =\int_{0}^{t}\int_{\R^N}|v_t(x,s)|^p\psi_1(x,s)dx \, ds + \e C_1(f_1,g_1),
\end{array}
\end{equation}
where 
\begin{equation}\label{Cfg}
C_1(f_1,g_1):=\int_{\R^N}\big[\big(\mu_1\rho_1(0)-\rho'_1(0)\big)f_1(x)+\rho_1(0)g_1(x)\big]\phi(x)dx.
\end{equation}
Now, using the definition of $\rho_1$, given by \eqref{lmabdaK}, and \eqref{lambda'lambda}, we deduce that
\begin{equation}\label{Cfg1}
\mu_1\rho_1(0)-\rho'_1(0)=\frac{\mu_1-1-\sqrt{\de_1}}{2}K_{\frac{\sqrt{\de_1}}{2}}(1)+K_{\frac{\sqrt{\de_1}}2 +1}(1).
\end{equation}
Therefore, we obtain that
\begin{equation}\label{Cfg}
C_1(f_1,g_1)=K_{\frac{\sqrt{\de_1}}{2}}(1) \int_{\R^N} \big[\frac{\mu_1-1-\sqrt{\de_1}}{2}f_1(x)+g_1(x)\big]\phi(x)dx +K_{\frac{\sqrt{\de_1}}2 +1}(1)\int_{\R^N}f_1(x)\phi(x)dx.
\end{equation}
Thanks to the hypotheses in Theorem \ref{blowup}, namely the positivity of the initial data and \eqref{hypfg}, the  constant $C_1(f_1,g_1)$ is positive.\\
Recall \eqref{F1def}  and \eqref{test11}, the equation  \eqref{energysol2-bis2} gives
\begin{equation}
\begin{array}{l}\label{eq6}
\d G_1'(t)+\Gamma_1(t)G_1(t)=\int_0^t\int_{\R^N}|v_t(x,s)|^p \psi_1(x,s)dx \, ds +\e \, C(f_1,g_1),
\end{array}
\end{equation}
where 
\begin{equation}\label{gamma}
\Gamma_1(t):=\frac{\mu_1}{1+t}-2\frac{\rho'_1(t)}{\rho_1(t)}.
\end{equation}
Now, we multiply  \eqref{eq6} by $\frac{(1+t)^{\mu_1}}{\rho^2_1(t)}$ and integrate over $(0,t)$, we deduce that
\begin{align}\label{est-G1}
 G_1(t)
\ge \frac{G_1(0)}{\rho^2_1(0)}\frac{\rho^2_1(t)}{(1+t)^{\mu_1}}+{\e}C_1(f_1,g_1)\frac{\rho^2_1(t)}{(1+t)^{\mu_1}}\int_0^t\frac{(1+s)^{\mu_1}}{\rho^2_1(s)}ds.
\end{align}
Observing that $\d G_1(0)=\ep K_{\frac{\sqrt{\de_1}}{2}}(\eta_0)\int_{\R^N}f_1(x) \phi(x)dx>0$ and employing \eqref{lmabdaK}, the estimate \eqref{est-G1} yields
\begin{align}\label{est-G1-1}
 G_1(t)
\ge {\e}C_1(f_1,g_1)(1+t)K^2_{\frac{\sqrt{\de_1}}{2}}(\eta_0(t+1))\int^t_{t/2}\frac{1}{(1+s)K^2_{\frac{\sqrt{\de_1}}{2}}(\eta_0(s+1))}ds.
\end{align}
Thanks to \eqref{Kmu}, we deduce the existence of $T_0^1=T_0^1(\mu_1,\nu_1,\eta_0)>1$ such that, for all $t \ge T_0^1$,
\begin{equation}
\left\{
\begin{array}{l}\label{est-double}
(1+t)K^2_{\frac{\sqrt{\de_1}}{2}}(\eta_0(t+1))>\frac{\pi}{4\eta_0} e^{-2\eta_0(t+1)}, \\ \text{and}  \\  (1+t)^{-1}K^{-2}_{\frac{\sqrt{\de_1}}{2}}(\eta_0(t+1))>\frac{\eta_0}{\pi} e^{2\eta_0(t+1)}.
\end{array}
\right.
\end{equation}
Combining \eqref{est-G1-1} and \eqref{est-double}, we infer that
\begin{align}\label{est-G1-2}
 G_1(t)
\ge \frac{\e}{4\eta_0}C_1(f_1,g_1)e^{-2\eta_0 t}\int^t_{t/2}e^{2\eta_0 s}ds\ge \frac{\e}{8\eta_0}C_1(f_1,g_1)e^{-2\eta_0 t}(e^{2\eta_0 t}-e^{\eta_0 t}), \ \forall \ t \ge T_0^1.
\end{align}
Finally, using $e^{2\eta_0 t}>2e^{\eta_0 t}, \forall \ t \ge 1$, we conclude that
\begin{align}\label{est-G1-3}
 G_1(t)
\ge \frac{\e}{16\eta_0}C_1(f_1,g_1), \ \forall \ t \ge T_0^1.
\end{align}
Similarly, we have an analogous estimate to \eqref{est-G1-3} for $G_2(t)$, namely
\begin{align}\label{est-G1-3-2}
 G_2(t)
\ge \frac{\e}{16\eta_0}C_2(f_2,g_2), \ \forall \ t \ge T_0^2.
\end{align}
Hence, it suffices to set $T_0=\max(T_0^1, T_0^2)$ to achieve the proof of Lemma \ref{F1-n}.
\end{proof}

In the following, we will prove a lower bound for the functional $\tilde{G}_i(t)$, defined by  \eqref{F2def}. This will be the subject of the next lemma.
\begin{lem}\label{F11}
Suppose that the assumptions in Theorem \ref{blowup} are fulfilled. Let $(u,v)$ be an energy solution of \eqref{energysol2}-\eqref{energysol3}. Then, for $i=1,2$, there exists $T_1=T_1(\mu_i,\nu_i, \eta_0)>1$ such that
\begin{equation}
\label{F2postive}
\tilde{G}_i(t)\ge C_{\tilde{G}_i}\, \e, 
\quad\text{for all}\ t  \in [T_1,T),
\end{equation}
where $C_{\tilde{G}_i}$ is a positive constant which depends on $f_i$, $g_i$, $N, R$, $\mu_i$ and $\eta_0$.
\end{lem}
\begin{proof}
The proof of the lemma will be carried out for $u$. The same conlcusion can be similarly perfomed for $v$. \\
Let $t \in [0,T)$. Thanks to the definition of $G_1$ and  $\tilde{G}_1$, given  by \eqref{F1def} and  \eqref{F2def}, respectively, and using \eqref{test11} together with the following identity
 \begin{equation}\label{def23}\d \frac{d \tilde{G}_1}{dt}(t) -\frac{\rho_1'(t)}{\rho_1(t)}G_1(t)= \tilde{G}_1(t),\end{equation}
 the equation  \eqref{eq6} gives
\begin{equation}
\begin{array}{l}\label{eq5bis}
\d \tilde{G}_1(t)+\left(\frac{\mu_1}{1+t}-\frac{\rho_1'(t)}{\rho_1(t)}\right)G_1(t)\\
=\d \int_0^t\int_{\R^N}|v_t(x,s)|^p\psi_1(x,s)dx \, ds +\e \, C_1(f_1,g_1).
\end{array}
\end{equation}
Now, taking the time-derivative of the  equation \eqref{eq5bis}, we infer that
\begin{equation}\label{F1+bis}
\begin{array}{l}
\d \frac{d \tilde{G}_1}{dt}(t)+\left(\frac{\mu_1}{1+t}-\frac{\rho_1'(t)}{\rho_1(t)}\right)\frac{d G_1}{dt}(t)-\left(\frac{\mu_1}{(1+t)^2}+\frac{\rho_1''(t)\rho_1(t)-(\rho_1'(t))^2}{\rho_1^2(t)}\right)G_1(t) \\
\d = \int_{\R^N}|v_t(x,t)|^p\psi_1(x,t)dx.
\end{array}
\end{equation}
Employing  \eqref{lambda} and   \eqref{def23}, the equation  \eqref{F1+bis} yields
\begin{align}\label{F1+bis2}
\d \frac{d \tilde{G}_1}{dt}(t)+\left(\frac{\mu_1}{1+t}-\frac{\rho_1'(t)}{\rho_1(t)}\right)\tilde{G}_1-\left(\eta_0^2-\frac{\nu_1^2}{(1+t)^2}\right)G_1(t) 
 = \int_{\R^N}|v_t(x,t)|^p\psi_1(x,t)dx.
\end{align}
Using \eqref{gamma}, the definition of $\Gamma_1(t)$, we conclude that 
\begin{equation}\label{G2+bis3}
\begin{array}{c}
\d \frac{d \tilde{G}_1}{dt}(t)+\frac{3\Gamma_1(t)}{4}\tilde{G}_1(t)\ge\Sigma_1^1(t)+\Sigma_1^2(t)+\Sigma_1^3(t),
\end{array}
\end{equation}
where 
\begin{equation}\label{sigma1-exp}
\begin{array}{rl}
\Sigma_1^1(t):=&\d \left(-\frac{\rho_1'(t)}{2\rho_1(t)}-\frac{\mu_1}{4(1+t)}\right)\left(\tilde{G}_1(t)+\left(\frac{\mu_1}{1+t}-\frac{\rho_1'(t)}{\rho_1(t)}\right)G_1(t)\right),
\end{array} 
\end{equation}
\begin{equation}\label{sigma2-exp}
\Sigma_1^2(t):=\d \left(\eta_0^2-\frac{\nu_1^2}{(1+t)^2}+\left(\frac{\rho_1'(t)}{2\rho_1(t)}+\frac{\mu_1}{4(1+t)}\right) \left(\frac{\mu_1}{1+t}-\frac{\rho_1'(t)}{\rho_1(t)}\right) \right)  G_1(t),
\end{equation}
and
\begin{equation}\label{sigma3-exp}
\Sigma_1^3(t):=  \int_{\R^N}|v_t(x,t)|^p\psi_1(x,t) dx.
\end{equation}
Combining the use of \eqref{eq5bis} and \eqref{lambda'lambda1}, we have the existence of $T^1_1=T^1_1(\mu_1) \ge T_0$ such that
\begin{equation}\label{sigma1}
\d \Sigma_1^1(t) \ge  \frac{\eta_0 \e}{4} \, C_1(f_1,g_1)+\frac{\eta_0}{4}\int_{0}^t \int_{\R^N}|v_t(x,s)|^p\psi_1(x,s)dx ds, \quad \forall \ t \ge T^1_1. 
\end{equation}
Employing Lemma \ref{F1} and \eqref{lambda'lambda1}, one can obtain the existence of a time  $\tilde{T}^1_1=\tilde{T}^1_1(\mu_1) \ge T^1_1$ for which we have
\begin{equation}\label{sigma2}
\d \Sigma_1^2(t) \ge 0, \quad \forall \ t  \ge  \tilde{T}^1_1. 
\end{equation}
Gathering all the above results, namely \eqref{G2+bis3}, \eqref{sigma3-exp}, \eqref{sigma1} and \eqref{sigma2}, we end up with the following estimate
\begin{equation}\label{G2+bis41}
\begin{array}{rcl}
\d \frac{d \tilde{G}_1}{dt}(t)+\frac{3\Gamma_1(t)}{4}\tilde{G}_1(t) &\ge& \d \frac{\eta_0 \e}{4} \, C_1(f_1,g_1)+\frac{\eta_0}{4}\int_{0}^t \int_{\R^N}|v_t(x,s)|^p\psi_1(x,s)dx ds \\ &+&\d \int_{\R^N}|v_t(x,t)|^p\psi_1(x,t) dx, \quad \forall \ t  \ge  \tilde{\tilde{T}}^1_1.
\end{array}
\end{equation}
At this level, we can eliminate  the nonlinear terms\footnote{ In fact, for a subsequent use in the proof of the main result, we choose here to keep the nonlinear terms up to this step in our computations. Otherwise,  omitting  the nonlinear terms can be done earlier in the proof of this lemma.} and we write 
\begin{equation}\label{G2+bis4}
\begin{array}{rcl}
\d \frac{d \tilde{G}_1}{dt}(t)+\frac{3\Gamma_1(t)}{4}\tilde{G}_1(t) &\ge& \d \frac{\eta_0 \e}{4} \, C_1(f_1,g_1), \quad \forall \ t  \ge  \tilde{\tilde{T}}^1_1.
\end{array}
\end{equation}
Integrating \eqref{G2+bis4} over $(\tilde{\tilde{T}}^1_1,t)$ after multiplication by $\frac{t^{3\mu_1/4}}{\rho_1^{3/2}(t)}$, we obtain
\begin{align}\label{est-G111-new}
 \tilde{G}_1(t)
&\ge \tilde{G}_1(\tilde{\tilde{T}}^1_1)\frac{(1+\tilde{\tilde{T}}^1_1)^{3\mu_1/4}}{\rho_1^{3/2}(\tilde{\tilde{T}}^1_1)}\frac{\rho_1^{3/2}(t)}{t^{3\mu_1/4}}\\&+\frac{\eta_0 \e}{4} \, C_1(f_1,g_1)\frac{\rho_1^{3/2}(t)}{t^{3\mu_1/4}}\int_{\tilde{\tilde{T}}^1_1}^t\frac{(1+s)^{3\mu_1/4}}{\rho_1^{3/2}(s)}ds, \quad \forall \ t  \ge  \tilde{\tilde{T}}^1_1.\nonumber
\end{align}
Recall that $\tilde{G}_1(t)=\rho_1(t)e^{t}\tilde{F}_1^{\eta_0}(t)$, where $\tilde{F}_1^{\eta_0}(t)$ is given by \eqref{F2def-old}, and using Lemma \ref{lemF2} we deduce that $\tilde{G}_1(t) \ge 0$ for all $t \in [0,T)$.\\
Hence, the fact that $\tilde{G}_1(t)$ is nonneagative together with the definition of   $\rho_1(t)$, given by  \eqref{lmabdaK}, yield
\begin{align}\label{est-G1111}
 \tilde{G}_1(\tilde{\tilde{T}}^1_1)\frac{(1+\tilde{\tilde{T}}^1_1)^{3\mu_1/4}}{\rho_1^{3/2}(\tilde{\tilde{T}}^1_1)}\frac{\rho_1^{3/2}(t)}{t^{3\mu_1/4}}
\ge 0, \quad \forall \ t \in [0,T).
\end{align}
Using \eqref{lmabdaK}, \eqref{est-double} and \eqref{est-G1111}, the estimate \eqref{est-G111-new} implies that
\begin{equation}\label{est-G2-12}
 \tilde{G}_1(t)
\ge   C\,{\e}e^{-3t/2} \int^t_{t/2}e^{3s/2}ds, \quad \text{for all} \ t \ge T^1_1:=2\tilde{\tilde{T}}^1_1.
\end{equation}
Consequently, we see that
\begin{align}\label{est-G1-2}
 \tilde{G}_1(t)
\ge  C\,{\e}, \quad \forall \ t \ge T^1_1.
\end{align}
Note that,  similarly  for $v$, we obtain the existence of $T^2_1=T^2_1(\mu_2)>1$. Finally, by setting $T_1=\max(T^1_1,T^2_1)$, we conclude the proof of Lemma \ref{F11}.
\end{proof}

\section{Proof of Theorem \ref{blowup}.}\label{sec-ut}

In this section we will prove Theorem \ref{blowup}. For that purpose, we will make use of the  results obtained in Section \ref{aux}. In fact, thanks to the invariance of the linear problem associated with \eqref{G-sys} and the coercive properties of $\tilde{G}_i(t)$, as stated  in Lemma \ref{F11}, we will show the blow-up result of  \eqref{G-sys}. Note that  the techniques used in our previous works \cite{Our3, Our4} cannot be entirely followed here. By introducing some new functionals $L_1(t)$ and $L_2(t)$ (see \eqref{L1} and \eqref{L2} below), which verify two integral inequalities similar to the ones in \cite{Our4}, we improve the blow-up result in \cite{Palmieri} for the solution of \eqref{G-sys}.  \\

Let
\begin{equation}\label{L1}
L_1(t):=
\frac{1}{8}\int_{T_2}^t  \int_{\R^N}|v_t(x,s)|^p\psi_1(x,s)dx ds
+\frac{C_3 \e}{8},
\end{equation}
and
\begin{equation}\label{L2}
L_2(t):=
\frac{1}{8}\int_{T_2}^t  \int_{\R^N}|u_t(x,s)|^q\psi_2(x,s)dx ds
+\frac{C_3 \e}{8},
\end{equation}
where $C_3=\min(C_1(f_1,g_1)/4,C_2(f_2,g_2)/4,8C_{\tilde{G}_1},8C_{\tilde{G}_2})$ (see Lemma  \ref{F11} for the constants $C_{\tilde{G}_1}$ and $C_{\tilde{G}_2}$) and $T_2:=T_2(\mu_1,\mu_2)>T_1$ is a positive time such that $\frac{\eta_0}{4}-\frac{3\Gamma_i(t)}{32}>0$ and $\Gamma_i(t)>0$, for i=1,2, for all $t \ge T_2$; see  \eqref{gamma} and \eqref{lambda'lambda1}. \\
Now, we introduce
$$\mathcal{F}_i(t):= \tilde{G}_i(t)-L_i(t), \quad \forall \ i=1,2.$$
Hence, thanks to \eqref{G2+bis41}, we see that $\mathcal{F}_1$ satisfies
\begin{equation}\label{G2+bis6}
\begin{array}{rcl}
\d \mathcal{F}^{'}_1(t)+\frac{3\Gamma_1(t)}{4}\mathcal{F}_1(t) &\ge& \d \left(\frac{\eta_0}{4}-\frac{3\Gamma_1(t)}{32}\right)\int_{T_2}^t \int_{\R^N}|v_t(x,s)|^p\psi_1(x,s)dx ds\vspace{.2cm}\\ &+&  \d \frac{7}{8}\int_{\R^N}|v_t(x,t)|^p\psi_1(x,t) dx+C_3 \left(\eta_0-\frac{3\Gamma_1(t)}{32}\right) \e\\
&\ge&0, \quad \forall \ t \ge T_2.
\end{array}
\end{equation}
After multiplying  \eqref{G2+bis6} by $\frac{t^{3\mu_1/4}}{\rho_1^{3/2}(t)}$ and integrating over $(T_2,t)$, we obtain
\begin{align}\label{est-G111}
 \mathcal{F}_1(t)
\ge \mathcal{F}_1(T_2)\frac{(T_2)^{3\mu/4}}{\rho_1^{3/2}(T_2)}\frac{\rho_1^{3/2}(t)}{t^{3\mu_1/4}}, \ \forall \ t \ge T_2,
\end{align}
where $\rho_1(t)$ is defined by \eqref{lmabdaK}.\\
Using Lemma \ref{F11} and  $C_3=\min(C_1(f_1,g_1)/4,C_2(f_2,g_2)/4,8C_{\tilde{G}_1},8C_{\tilde{G}_2}) \le 8C_{\tilde{G}_1}$, one can see that $\d \mathcal{F}_1(T_2)=\GG(T_2)-\frac{C_3 \e}{8} \ge C_{\tilde{G}_1}\e -\frac{C_3 \e}{8}\ge 0$. \\
Consequently, we infer that
\begin{equation}
\label{G2-est}
\GG(t)\geq L_1(t), \ \forall \ t \ge T_2.
\end{equation}
In a similar way, we have an analogous lower bound for $\HH (t)$, that is
\begin{equation}
\label{G2-est-bis}
\HH(t)\geq L_2(t), \ \forall \ t \ge T_2.
\end{equation}
Employing the H\"{o}lder's inequality together with the estimates \eqref{psi} and \eqref{F2postive}, a lower bound for the nonlinear term can written as
\begin{equation}\label{vt-pho}
\begin{array}{rcl}
\d \int_{\R^N}|v_t(x,t)|^p\psi_1(x,t)dx &\geq&\d (\HH(t))^p\left(\int_{|x|\leq t+R}(\psi_2(x,t))^{\frac{p}{p-1}}(\psi_1(x,t))^{\frac{-1}{p-1}}dx\right)^{-(p-1)} \vspace{.2cm}\\ &\geq& C (\HH(t))^p \rho_1(t)\rho_2^{-p}(t)e^{-(p-1)t}t^{-\frac{(N-1)(p-1)}2}.
\end{array}
\end{equation}
From \eqref{lmabdaK} and \eqref{est-double}, observe that
 \begin{equation}\label{pho-est}
 \d \rho_1(t)e^{t} \le C t^{\frac{\mu_1}{2}}, \ \forall \ t \ge T_0/2.
 \end{equation}
Note that similar estimate holds for $ \rho_2(t)$.\\
Combining  \eqref{pho-est} (and the equivalent estimate for $ \rho_2(t)$) and \eqref{vt-pho}, we deduce that  
\begin{equation}\label{4.9}
\d \int_{\R^N}|v_t(x,t)|^p\psi_1(x,t)dx \geq C t^{-\frac{(N-1)}{2}(p-1)+ \frac{\mu_1}{2}-\frac{\mu_2}{2}p}(\HH(t))^p, \ \forall \ t \ge T_2.
\end{equation}
Now, recall the definition of $L_1(t)$, given by \eqref{L1}, and injecting \eqref{G2-est-bis} in \eqref{4.9}, we conclude that
\begin{equation}
\label{inequalityfornonlinearin}
L_1^{'}(t)\geq C t^{-\frac{(N-1)}{2}(p-1)+ \frac{\mu_1}{2}-\frac{\mu_2}{2}p}(L_2(t))^p, \quad \forall \ t \ge T_2.
\end{equation}
Likewise, we have
\begin{equation}
\label{inequalityfornonlinearin2}
L_2^{'}(t)\geq C t^{-\frac{(N-1)}{2}(q-1)+ \frac{\mu_2}{2}-\frac{\mu_1}{2}q}(L_1(t))^q, \quad \forall \ t \ge T_2.
\end{equation}
A straightforward integration of \eqref{inequalityfornonlinearin} and \eqref{inequalityfornonlinearin2} on $(T_2, t)$ yields, respectively,
\begin{equation}
L_1(t)\geq \frac{C_3 \e}{8}+C \int_{T_2}^t  (1+s)^{-\frac{(N-1)}{2}(p-1)+ \frac{\mu_1}{2}-\frac{\mu_2}{2}p}(L_2(s))^p ds, \quad \forall \ t \ge T_2,
\end{equation}
and
\begin{equation}
L_2(t)\geq \frac{C_3 \e}{8}+C \int_{T_2}^t  (1+s)^{-\frac{(N-1)}{2}(q-1)+ \frac{\mu_2}{2}-\frac{\mu_1}{2}q}(L_1(s))^q ds, \quad \forall \ t \ge T_2.
\end{equation}
Observe  that $\d \frac{1}{T_2}(T_2+s)\le 1+s \le T_2+s$ for all $s \in (T_2, t)$, because $T_2>1$, we infer that
\begin{equation}
\label{integ-ineq}
L_1(t)\geq \frac{C_3 \e}{8}+C \int_{T_2}^t  (T_2+s)^{-\frac{(N-1)}{2}(p-1)+ \frac{\mu_1}{2}-\frac{\mu_2}{2}p}(L_2(s))^p ds, \quad \forall \ t \ge T_2,
\end{equation}
and
\begin{equation}
\label{integ-ineq2}
L_2(t)\geq \frac{C_3 \e}{8}+C \int_{T_2}^t  (T_2+s)^{-\frac{(N-1)}{2}(q-1)+ \frac{\mu_2}{2}-\frac{\mu_1}{2}q}(L_1(s))^q ds, \quad \forall \ t \ge T_2.
\end{equation}
At this level, the remaining part of the proof is the same as the one in \cite[Sections 4.2 and 4.3]{Palmieri}. More precisely, here \eqref{integ-ineq} (resp. \eqref{integ-ineq2} corresponds to (25) (resp. (26)) in \cite{Palmieri}. Nevertheless,  in the present work the shift of the dimension $N$ is with  $\mu_i$ instead of $\sigma(\mu_i)$ in \cite{Palmieri}, where $\sigma(\mu_i)$ is defined by \eqref{sigma}.

This achieves the proof of Theorem \ref{blowup}.

\section{Appendix}
The aim of this appendix is to recall some properties of the function $\rho_i^{\eta}(t)$, for $i=1,2$, the solution of \eqref{lambda}. Mainly, we will use  the computations in \cite{Tu-Lin}. Hence, we can write the expression of  $\rho_i^{\eta}(t)$ as follows:
\begin{equation}\label{lmabdaK-A}
\rho^{\eta}_i(t)=(\eta (t+1))^{\frac{\mu_i+1}{2}}K_{\frac{\sqrt{\de_i}}{2}}(\eta  (t+1)), \ i=1,2,
\end{equation}
where 
$$K_{\nu}(t)=\int_0^\infty\exp(-t\cosh \zeta)\cosh(\nu \zeta)d\zeta,\ \nu\in \mathbb{R}.$$
From  the proof of  \cite[Lemma 2.1]{Tu-Lin}, one can see that
\begin{equation}\label{lambda'lambda}
\frac{1}{\rho^{\eta}_i(t)}\frac{d\rho^{\eta}_i(t)}{dt}=\frac{\mu_i+1+\sqrt{\de_i}}{2  (t+1)}-\eta \frac{K_{\frac{\sqrt{\de_i}}2 +1}(\eta  (t+1))}{K_{\frac{\sqrt{\de_i}}2}(\eta  (t+1))}, \ i=1,2.
\end{equation}
On the other hand,   the function $K_{\nu}(t), \nu \in \R,$ satisfies (\cite{Gaunt})
\begin{equation}\label{Kmu}
K_{\nu}(t)=\sqrt{\frac{\pi}{2t}}e^{-t} (1+O(t^{-1}), \quad \text{as} \ t \to \infty.
\end{equation}
A combination of \eqref{lambda'lambda} and \eqref{Kmu} yields 
\begin{equation}\label{lambda'lambda1}
\frac{1}{\rho^{\eta}_i(t)}\frac{d\rho^{\eta}_i(t)}{dt}=-\eta+O(t^{-1}), \quad \text{as} \ t \to \infty, \ i=1,2.
\end{equation}


\bibliographystyle{plain}

\begin{thebibliography}{20}


\bibitem{Dao-Reissig}
{T. A. Dao and M. Reissig,} 
{\it The interplay of critical regularity of nonlinearities in a weakly coupled system of semi-linear damped wave equations.} Journal of Differential Equations, Volume {\bf 299}, 25 October 2021, Pages 1--32.



\bibitem{Gaunt}
{R.E. Gaunt,} {\it  Inequalities for modified Bessel functions and their integrals.} J. Mathematical
Analysis and Applications, {\bf 420} (2014), 373--386.

\bibitem{Our3}
{M. Hamouda and M.A. Hamza}, {\it  A blow-up result for the wave equation  with localized initial data: the  scale-invariant damping and mass term with combined nonlinearities}. To appear in Journal of Applied Analysis and Computation.

\bibitem{Our}
{M. Hamouda and M.A. Hamza}, {\it  Blow-up for  wave equation  with the  scale-invariant damping and combined nonlinearities}. Math Meth. Appl. Sci. Volume {\bf 44}, Issue 1,  2021, Pages 1127--1136.

\bibitem{Our4}
{M. Hamouda and M.A. Hamza}, {\it  Improvement on the blow-up for  the weakly coupled wave equations  with  scale-invariant damping and time derivative nonlinearity}. To appear in Mediterranean Journal of Mathematics.

\bibitem{Our2}
{M. Hamouda and M.A. Hamza}, {\it  Improvement on the blow-up of the wave equation  with the  scale-invariant damping and combined nonlinearities}. Nonlinear Anal. Real World Appl. Volume {\bf 59}, 2021, 103275, ISSN 1468--1218, https://doi.org/10.1016/j.nonrwa.2020.103275.




\bibitem{Hidano1}
{K. Hidano and K. Tsutaya}, {\it Global existence and asymptotic behavior of solutions for nonlinear wave
equations}, Indiana Univ. Math. J., {\bf 44} (1995), 1273--1305.

\bibitem{Hidano2}
{K. Hidano, C. Wang and K. Yokoyama}, {\it The Glassey conjecture with radially symmetric data}, J. Math.
Pures Appl.,  (9) {\bf 98} (2012),  no. 5, 518--541.


\bibitem{Ikeda-sys}
M. Ikeda, M. Sobajima and K. Wakasa,
{\it Blow-up phenomena of semilinear wave equations and their weakly coupled systems.} 
J. Differential Equations, {\bf 267} (2019), no. 9, 5165–5201.

\bibitem{John1}
{F. John}, {\it Blow-up for quasilinear wave equations in three space dimensions}, Comm. Pure Appl. Math., {\bf 34} (1981), 29--51.


\bibitem{Kubo}
{H. Kubo, K. Kubota and H. Sunagawa,} 
{\it Large time behavior of solutions to semilinear systems of wave equations.}
Math. Ann. {\bf 335} (2006), no. 2, 435--478.

\bibitem{LT}{N.-A. Lai and H. Takamura},
{\it Blow-up for semilinear damped wave equations with subcritical exponent in the scattering case.}  Nonlinear Anal. {\bf 168} (2018), 222--237.

\bibitem{LT2}{N.-A. Lai and H. Takamura},
{\it Nonexistence of global solutions of nonlinear wave equations with weak time-dependent damping related to Glassey's conjecture.} 
Differential Integral Equations, {\bf 32} (2019), no. 1-2, 37--48.

\bibitem{LT3}{N.-A. Lai and H. Takamura},
{\it Nonexistence of global solutions of wave equations with weak time-dependent damping and combined nonlinearity.} Nonlinear Anal. Real World Appl. {\bf 45} (2019), 83--96.



\bibitem{Palmieri1}{A. Palmieri,} {\it A note on a conjecture for the critical curve of a weakly coupled system of semilinear wave equations with scale-invariant lower order terms. }Vol. {\bf 43}, Issue 11 (2020), 6702--6731.

\bibitem{Palmieri-Takamura-arx}{A. Palmieri and H. Takamura},
{\it  Nonexistence of global solutions for a weakly coupled system of semilinear damped wave equations in the scattering case with mixed nonlinear terms.}  	Nonlinear Differ. Equ. Appl. 27, 58 (2020).

\bibitem{Palmieri-Takamura}{A. Palmieri and H. Takamura},
{\it  Nonexistence of global solutions for a weakly coupled system of semilinear damped wave equations of derivative type in the scattering case.} Mediterr. J. Math. {\bf 17} (2020), no. 1, Paper No. 13, 20 pp. 


\bibitem{Palmieri}{A. Palmieri and Z. Tu},
 {\it A blow-up result for a semilinear wave equation with scale-invariant damping and mass and nonlinearity of derivative type.} Calc. Var. {\bf 60}, 72 (2021). https://doi.org/10.1007/s00526-021-01948-0.


\bibitem{Palmieri-Tu}{A. Palmieri and Z. Tu}, 
{\it Lifespan of semilinear wave equation with scale invariant dissipation and mass and sub-Strauss power nonlinearity.} 
J. Math. Anal. Appl. {\bf 470} (2019), no. 1, 447--469.






\bibitem{Sideris}
{T. C. Sideris}, {\it Global behavior of solutions to nonlinear wave equations in three space dimensions}, Comm. Partial Differential Equations, {\bf 8} (1983), no. 12, 
1291--1323.


\bibitem{Tu-Lin1}
{Z. Tu,  and  J. Lin}, {\it  A note on the blowup of scale invariant damping wave equation with sub-Strauss exponent, preprint, arXiv:1709.00866v2, 2017. }



\bibitem{Tu-Lin}
{Z. Tu,  and  J. Lin}, {\it  
Life-span of semilinear wave equations with scale-invariant damping: critical Strauss exponent case. }
Differential Integral Equations, {\bf 32} (2019), no. 5-6, 249--264.

\bibitem{Tzvetkov}
{N. Tzvetkov}, {\it Existence of global solutions to nonlinear massless Dirac system and wave equation
with small data}, Tsukuba J. Math., {\bf 22} (1998), 193--211.



 
\bibitem{Xu}{W. Xu,} 
{\it   Blowup for systems of semilinear wave equations with small initial data.}
J. Partial Differential Equations {\bf 17} (2004), no. 3, 198--206.

\bibitem{YZ06}
{B. Yordanov and Q. S. Zhang}, {\it Finite time blow up for critical wave equations in high dimensions}, J. Funct. Anal., {\bf 231} (2006), 361--374.



\bibitem{Zhou1}
{Y. Zhou}, {\it Blow-up of solutions to the Cauchy problem for nonlinear wave equations}, Chin. Ann. Math., {\bf 22B}  (3) (2001), 275--280.




\end{thebibliography}

\end{document}